\documentclass[10pt]{amsart}

\usepackage{amssymb, amsmath, amsfonts}
\usepackage{amscd}
\usepackage[all]{xy}
\numberwithin{equation}{section}

\newcommand{\pp}{\mathbb P}
\newcommand{\cc}{\mathbb C}
\newcommand{\ox}{\mathcal{O}_X}

\newcommand{\A}{\mathbb A}
\newcommand{\V}{\mathcal V}
\newcommand{\Sym}{\mathrm{Sym}}
\newcommand{\rank}{\mathrm{rk}\,}

\newcommand{\Hom}{\mathrm{Hom}}
\newcommand{\Ker}{\mathrm{Ker}}
\newcommand{\Sec}{\mathrm{Sec}}
\newcommand{\Pic}{\mathrm{Pic}}
\newcommand{\Gr}{\mathrm{Gr}}
\newcommand{\OG}{\mathrm{OG}}

\newcommand{\tse}{\tilde{S_e}}
\newcommand{\tsu}{\tilde{U}_e}
\newcommand{\tE}{\tilde{E}}
\newcommand{\tF}{\tilde{F}}

\newtheorem{theorem}{{\textbf Theorem}}[section]
\newtheorem{prop}[theorem]{{\textbf Proposition}}
\newtheorem{cor}[theorem]{{\textbf Corollary}}
\newtheorem{lem}[theorem]{{\textbf Lemma}}

\newtheorem{rmk}[theorem]{{\textbf Remark}}
\newenvironment{remark}{\begin{rmk}\rm}{\end{rmk}}

\title[Lagrangian subbundles of odd rank orthogonal bundles]{Lagrangian subbundles of orthogonal bundles \\of odd rank over an algebraic curve}
\author{Insong Choe and George H. Hitching}

\begin{document}
\thispagestyle{empty}

\begin{abstract} An orthogonal bundle over a curve has an isotropic Segre invariant determined by the maximal degree of a Lagrangian subbundle. This invariant and the induced stratifications on moduli spaces of orthogonal bundles, were studied for bundles of even rank in \cite{CH3}. In this paper, we obtain analogous results for bundles of odd rank. We obtain a sharp upper bound on the isotropic Segre invariant. We show the irreducibility of the induced strata on the moduli spaces of orthogonal bundles of odd rank, and compute their dimensions. As a key ingredient of the proofs, we study the correspondence between Lagrangian subbundles of orthogonal bundles of even and odd rank.
\end{abstract}

\maketitle

\section{Introduction}

The aim of this paper is to study Lagrangian subbundles of orthogonal bundles of odd rank over an algebraic curve. The results in this paper, together with those in the earlier paper \cite{CH3} on orthogonal bundles of even rank, will provide a complete picture on the Segre-type stratifications corresponding to Lagrangian subbundles on moduli spaces of orthogonal bundles.

Let $X$ be a smooth irreducible algebraic curve of genus $g \ge 2$ over $\cc$. A vector bundle $V$ is called \textit{orthogonal} if there is a nondegenerate symmetric bilinear form $\omega \colon V \otimes V \to \ox$. Equivalently, $V$ is orthogonal if it admits an isomorphism $\omega \colon V \xrightarrow{\sim} V^*$ which is symmetric in the sense that $\omega$ coincides with its transpose $^t \omega \colon (V^*)^* \to V^*$. A subbundle $E$ of $V$ is called \textit{isotropic} if $E_x$ is an isotropic subspace of $V_x$ for each $x \in X$. An isotropic subbundle of $V$ may have rank at most $\left\lfloor \frac{\rank V}{2} \right\rfloor$, and is called \textit{Lagrangian} if it has rank equal to $\left\lfloor \frac{\rank V}{2} \right\rfloor$. The case when $\rank V$ is even was studied in \cite{CH3}, where the terminology ``isotropic subbundle of half rank'' was used instead of ``Lagrangian subbundle''.

Generalizing the Segre invariant for vector bundles of rank 2 (or ruled surfaces), we define
\[
t(V) \ := \ -2 \max \{\deg E : E \text{ a Lagrangian subbundle of } V \}.
\]
The coefficient 2 is inserted here to match the Segre invariant for vector bundles of rank 2. A Lagrangian subbundle $E$ of $V$ is called a \textit{maximal Lagrangian subbundle} if $t(V) = -2 \deg E$. The invariant $t(V)$ define natural stratifications on the moduli spaces of orthogonal bundles, which are the main object of study in this paper.

The analogous situation for vector bundles is well understood: The upper bound on the Segre invariant of a vector bundle was computed by Hirschowitz \cite{Hir} (see also \cite{CH1}). Furthermore, the stratification on the moduli space $U_X(r,d)$ of semistable vector bundles of rank $r$ and degree $d$ was studied in Brambila-Paz--Lange \cite{BPL} and Russo--Teixidor i Bigas \cite{RT}. The cases of symplectic bundles and orthogonal bundles of even rank were studied by the present authors in \cite{CH2, CH3}.

To state our main results, let us discuss the moduli space of orthogonal bundles over $X$ of odd rank $2n+1$. We exploit the relation between principal $\mathrm{SO}_{2n+1} \cc$-bundles and orthogonal bundles. 
By Serman \cite{Ser}, the forgetful map from the moduli space of semistable principal $\mathrm{SO}_{2n+1}$-bundles to that of $\mathrm{SL}_{2n+1}$-bundles is a (closed) embedding. The image is identified with the orthogonal locus $MO_X(2n+1)$ inside the moduli space $SU_X(2n+1, \ox)$ of semistable vector bundles over $X$ of rank $2n+1$ with trivial determinant.

The moduli space $MO_X(2n+1)$ has two connected components, distinguished by the second Stiefel--Whitney class $w_2(V ) \in H^2 (X, \mathbb{Z} / 2\mathbb{Z})$. Serre \cite{Serre} showed that
\[
w_2(V) \equiv h^0(X, V \otimes \kappa) + h^0(X, \kappa) \mod 2,
\]
where $\kappa \in \mathrm{Pic}^{g-1}(X)$ is any theta-characteristic. We will denote the two components by $MO_X(2n+1)^\pm$, where $MO_X(2n+1)^+$ is the component containing the trivial orthogonal bundle. The tangent space of $MO_X(2n+1)$ at $V$ is given by $H^1(X, \wedge^2 V)$. Hence $\dim MO_X(2n+1)^\pm = n(2n+1)(g-1)$. 

Firstly, we prove the following:

\newtheorem*{theorem1}{Theorem \ref{thm1}}
\begin{theorem1} \begin{enumerate}
\renewcommand{\labelenumi}{(\arabic{enumi})}
\item Let $E_1$ and $E_1$ be Lagrangian subbundles of an orthogonal bundle $V$. Then $\deg E_1$ and $\deg E_2$ have the same parity.
\item A semistable orthogonal bundle $V$ belongs to the component $MO_X(2n+1)^+$ (resp., $MO_X(2n+1)^-$) if and only if its Lagrangian subbundles have even degree (resp., odd degree). \qed \end{enumerate} 
\end{theorem1}
\noindent This easily follows from the corresponding result for $MO_X(2n)$ which was proven in \cite[Theorem 1.2]{CH3}, but we give another proof in Proposition \ref{parityeven} using the relationship between the two components of the orthogonal Grassmannian $\OG(n, 2n)$. 

Holla and Narasimhan \cite{HN} defined an invariant $s(V;P)$ more generally for arbitrary principal $G$-bundles $V$ with respect to a fixed parabolic subgroup $P$. When $G = \mathrm{O}_{2n+1}$ and $P$ is the maximal parabolic subgroup preserving a fixed isotropic subspace of dimension $n$,  one computes that $s(V;P) = \frac{1}{2}(n-1) \cdot  t(V)$.  
The upper bound on $s(V;P)$ given in \cite{HN} can then be interpreted as
\[
t(V) \  \le \ \frac{n(n+1)}{n-1} g.
\]
We obtain a sharp upper bound on $t(V)$, as follows:
\newtheorem*{cor1}{Corollary \ref{thm2}}
\begin{cor1}  
Any orthogonal bundle $V$ of rank $2n+1$ satisfies
\[ t(V) \ \le \ (n+1)(g-1)+3. \]
This bound is sharp in the sense that the two even numbers $t$ with
\[ (n+1)(g-1) \ \le \ t \ \le \ (n+1)(g-1) + 3 \]
correspond to the values of $t(V)$ for general $V \in MO_X(2n+1)^\pm$.  \qed
\end{cor1}
It is noteworthy that in most cases, this bound is strictly greater than Hirschowitz' bound on the Segre invariant of subbundles of vector bundles. This means that a general orthogonal bundle of rank $2n+1$ has the property that no maximal subbundle of rank $n$ is Lagrangian. For details, see Remark \ref{Hirsch}. A similar phenomenon for certain orthogonal bundles of even rank is described in \cite[Remark 5.5]{CH3}, but the difference of the bounds is smaller compared to the odd rank case. 
This is a contrast to the case of symplectic bundles: From \cite[Theorem 1.1 (1)]{CH3} it follows that the degree of a maximal Lagrangian subbundle of a general symplectic bundle of rank $2n$ coincides with the degree of a maximal rank $n$ subbundle of a general vector bundle of rank $2n$ and degree zero.

As mentioned, the invariant $t(V)$ induces a stratification on $MO_X(2n+1)^\pm$. For each even number $t$ with $0 < t \le (n+1)(g-1)+3$, we define
\[
MO_X(2n+1;t) \ := \ \{ V \in MO_X(2n+1) : t(V) = t \}.
\]
Since the invariant $t(V)$ is semicontinuous, these subloci are constructible sets. By Theorem \ref{thm1}, the loci $MO_X(2n+1;t_1)$ and $MO_X(2n+1;t_2)$ belong to the same component if and only if $t_1 \equiv t_2 \mod 4$. We show:

\newtheorem*{theorem2}{Theorem \ref{thm3}}
\begin{theorem2}  
For each even number $t$ with $0 < t \le (n+1)(g-1)+3$, the stratum $MO_X(2n+1;t)$ is nonempty and irreducible. For $t \le (n+1)(g-1)$, it has dimension $\frac{1}{2} n(3n+1)(g-1) +\frac{1}{2}nt$. \qed
\end{theorem2}

For each stratum, we prove the following result on the number of maximal Lagrangian subbundles: 

\newtheorem*{theorem4}{Theorem \ref{thm4}}
\begin{theorem4}  
Let $V$ be a general orthogonal bundle in $MO_X(2n+1;t)$. If $t <(n+1)(g-1)$, then $V$ has a unique maximal Lagrangian subbundle. If $t = (n+1)(g-1)$ (resp., $t > (n+1)(g-1)$) is even, $V $ has finitely many (resp., infinitely many) maximal Lagrangian subbundles.  \qed
\end{theorem4}
\noindent We also compute the dimension of the space of Lagrangian subbundles of a general $V \in MO_X(2n+1)$ when it is positive. 

All these results parallel the previous results in \cite{CH3} for orthogonal bundles of even rank. The proofs, however, are rather different. In the even rank case, the isotropic Segre invariant is determined by a geometric criterion \cite[\S 2.4]{CH3} analogous to Lange--Narasimhan \cite[Proposition 1.1]{LN}. It seems more difficult to give such a criterion directly in the odd rank case. Our strategy, instead, is as follows: 
As stated in Lemma \ref{OG}, there is a natural correspondence between Lagrangian subspaces of the orthogonal vector spaces $\cc^{2n+1}$ and $\cc^{2n+2}$. 
We describe a similar correspondence between Lagrangian subbundles of an orthogonal bundle of rank $2n+1$ and a certain orthogonal bundle of rank $2n+2$ (Lemma \ref{Lagsubbs}). Several statements can then be deduced from the even rank case, including a modified geometric condition on the Segre invariant (Lemma \ref{liftingcrit}), which suffices to prove the required statements.

This paper is organized as follows: In \S2, we quote some of the results on orthogonal bundles of even rank which are relevant in our context. In \S3, orthogonal extensions of odd rank are discussed, and some interplays between even and odd rank cases are explained. In \S4, parameter spaces of orthogonal extensions are constructed, and it is shown that a general point thereof represents a stable orthogonal bundle. Hence there is a (rational) moduli map from the parameter space of orthogonal extensions to the moduli space $MO_X(2n+1)$. In \S5, we study the condition  under which an orthogonal bundle admits two Lagrangian subbundles with a prescribed degree condition. This shows that the moduli map sends each parameter space to a dense subset of a Segre stratum. In \S6, the general fibers of this moduli map are studied and the dimension of the space of maximal Lagrangian subbundles of a general orthogonal bundle is computed.

\section{Orthogonal bundles of even rank}

In this section, we quote some results from \cite{CH3} and \cite{Hit1} on orthogonal bundles of even rank, which are relevant for our later discussion.\\
\\
\textbf{Notation:} Throughout, $X$ denotes a complex projective smooth curve of genus $g \ge 2$. If $0 \to E \to V \to F \to 0$ is an extension of vector bundles over $X$, we denote the class of $V$ in $H^1 (X, \Hom(F, E))$ by $[V]$.

\subsection{Orthogonal extensions of even rank} Let $F$ be a Lagrangian subbundle of an orthogonal bundle $V$ of rank $2n$. Since $V/F \ \cong \ (F^\perp)^* \ = \ F^*$, the bundle $V$ fits into the exact sequence $0 \to F \to V \to F^* \to 0$.
\begin{prop} \label{evenrankcrit} {\rm (}\cite[Criterion 2.1]{Hit1}{\rm )} \ \
Let $F$ be a simple bundle of rank $n$. An extension class $[V] \in H^1(X, \otimes^2 F)$ is induced by an orthogonal structure with respect to which $F$ is Lagrangian if and only if $[V]$ belongs to the subspace $H^1(X, \wedge^2 F)$. \qed
\end{prop}


\subsection{Lifting criterion for quotient line bundles}

The following result, which is a generalization of Lange--Narasimhan \cite[Proposition 1.1]{LN}, will be used in checking generality conditions. In \cite{LN}, the case of rank two extensions was discussed, for which $\pp E \cong X$.\\
\\
\textbf{Notation:} If $Z \subseteq \pp^N$ is a quasi-projective variety, then $\Sec^k Z$  denotes the secant variety given by the closure of the union of linear spaces spanned by $k$ general points  of $Z$.
\begin{prop} \label{linelift}
Let $E$ be a stable bundle of negative degree.
\begin{enumerate}
\renewcommand{\labelenumi}{(\arabic{enumi})}
\item {\rm(}\cite[Lemma 2.3]{CH3}{\rm)} There is a canonical rational map $\phi \colon \pp E \dashrightarrow \pp H^1(X, E)$ which is injective on a general fiber.
\item {\rm(}Adaptation of \cite[Theorem 4.4]{CH1}{\rm)} Let $F$ be an extension of $\ox$ by $E$. If there is a subsheaf $\ox(-D)$ of $\ox$ lifting to $F$ for some effective divisor $D$, then the class $[F]$ lies on the kernel of the surjection $H^1(X, E) \to H^1(X, E(D))$. In particular, it lies on $\Sec^d \pp E$, where $d = \deg D$.  \qed
\end{enumerate}
\end{prop}


\subsection{Two Lagrangian subbundles}

Suppose that an orthogonal bundle $V$ of rank $2n$ has two different Lagrangian subbundles $F$ and $\tilde{F}$.  Let $H$ be the locally free part of the intersection. Both subbundles $F$ and $\tilde{F}$ lie inside $H^\perp$, inducing a diagram:
 \begin{equation} \label{Hperp}
\begin{CD}
0 @>>> F @>>> H^\perp @>>> H^\perp /F @>>> 0 \\
 @. @AAA @AAA @AAA @. \\
0 @>>> H @>>> \tilde{F} @>>> \tilde{F}/H @>>> 0 
\end{CD}
\end{equation}
It is easy to check that $H^\perp /H$ is an orthogonal bundle of rank $2(n-r)$ and $F/H$ and $\tF/H$ are Lagrangian subbundles, yielding a diagram for some torsion sheaf $\tau$: 
 \begin{equation} \label{paritydiagram}
\begin{CD} 
0 @>>> F/H @>>> (\tF/H)^* @>>> \tau @>>> 0\\
@. @| @AAA @AAA @.\\
0 @>>> F/H @>>> H^\perp/H @>>> (F/H)^* @>>> 0 \\
 @. @.  @AAA @AAA @. \\
@. @. \tilde{F}/H @= \tilde{F}/H @.
\end{CD}
\end{equation}
Note that  $\deg (\tilde{F}/H) \le \deg (F/H)^*$. Therefore,
\begin{equation} \label{degh}
\deg H  \ge \frac{1}{2} (\deg F + \deg \tilde{F}).
\end{equation}
\begin{prop} \label{parityeven} {\rm(}\cite[Theorem 1.2]{CH3}{\rm)}
\begin{enumerate}
\renewcommand{\labelenumi}{(\arabic{enumi})}
\item Let $F$ and $\tF$ be Lagrangian subbundles of $V$. Then $\deg F$ and $\deg \tF$ have the same parity.
\item The Stiefel--Whitney class $w_2 (V)$ is trivial (resp., nontrivial) if and only if the Lagrangian subbundles of $V$ have even degree (resp., odd degree).
\item If $V$ is semistable, then $V$ belongs to $MO_X(2n)^+$ (resp., $MO_X(2n)^-$) if and only if its Lagrangian subbundles have even degree (resp., odd degree).
\end{enumerate}
\end{prop}
\begin{proof}
(2) and (3) are direct consequences of (1). The proof of (1) in \cite[Theorem 1.2]{CH3} was based on the observation that general two Lagrangian subbundles are related by elementary transformations associated to an antisymmetric principal part, which is of even degree. Here we will give another argument. 

Firstly, recall the fact that the orthogonal Grassmannian $\OG(n, 2n)$ consists of two disjoint irreducible components such that two Lagrangian subspaces $\mathsf{L}$ and $\mathsf{L}'$ belong to the same component if and only if $\dim (\mathsf{L} \cap \mathsf{L}') \equiv n \mod 2$ (see Lemma \ref{OG} (2)). The  Lagrangian subbundle $F$ yields a family of Lagrangian subspaces $F_x$ of $V_x$.  Therefore, if $F_{x_0}$ and $\tF_{x_0}$ belong to the same component for some $x_0 \in X$, then so do $F_x$ and $\tF_x$ for every $x \in X$. 

Let $H$ be the locally free part of $F_1 \cap F_2$ inside $V$ with $\rank H = r \ge 0$. By the above discussion, $\dim (F_x \cap \tF_x) \equiv r \mod 2$ for every $x \in X$.  Therefore, the torsion subsheaf of $\tau$ supported at $x$ is of even degree for each $x \in X$. Therefore, $\deg \tau$ is even. This shows that $\deg F$ and $\deg \tF$ have the same parity.
\end{proof}
\begin{cor} \label{rank1case}
Suppose that $\mathrm{rk}(F/H) = 1$ in (\ref{paritydiagram}). Then $\tau = 0$ and $F \cap \tF = H$, and $H^\perp /H \cong (F/H) \oplus (\tF/H)$.
\end{cor}
\begin{proof}
Under the assumption, we have
\[
n-1 \ \le \ \dim (F_x \cap \tF_x) \equiv n-1 \mod 2
\]
for each $x \in X$. Therefore, $\dim (F_x \cap \tF_x) = n-1$ for each $x \in X$. This implies that $\tau = 0$ and $F \cap \tF = H$. Since all orthogonal bundles of rank 2 are direct sums of the form $M \oplus M^{-1}$ where $M$ is a line bundle, the middle sequence in (\ref{paritydiagram}) splits.
\end{proof}
Let $p_H$ denote the surjection $H^1(X, \wedge^2 F) \ \to \ H^1(X, \wedge^2(F/H))$ induced by the quotient map $F \to F/H$. We recall the following results on this situation:

\begin{prop} \label{quote}
Let $F$ be a general stable bundle of degree $-f < 0$. Let $\Gr(2,F)$ be the Grassmannian bundle over $X$ of 2-dimensional subspaces of the fibers of $F$.
\begin{enumerate}
\renewcommand{\labelenumi}{(\arabic{enumi})}
\item {\rm(}\cite[Lemma 2.3]{CH3}{\rm)} \ There is a canonical rational map $\phi_a \colon \Gr(2, F) \dashrightarrow \pp H^1(X, \wedge^2 F)$, which is injective on a general fiber.
\item {\rm(}\cite[Criterion 2.11]{CH3}{\rm)} \ Let $V$ be an orthogonal bundle admitting $F$ as a Lagrangian subbundle, with extension class $[V] \in H^1(X, \wedge^2 F)$. Then $V$ has another Lagrangian subbundle $\tilde{F}$ of degree $-\tilde{f}$ inducing a diagram (\ref{Hperp}) for some $H$ of rank $\le n-2$ and  degree $-h$ if and only if  
\[
p_H \left( [V] \right) \ \in \ \Sec^{k}\Gr(2, F/H),
  \] 
where $k := \frac{1}{2} (f+\tilde{f}-2h) \ge 0 $.   \qed
\end{enumerate} 
\end{prop}

\begin{remark} We mention some special cases: When $H = 0$, the statement yields a criterion for isotropic liftings of elementary transformations of $F^*$. If $\rank(F/H) = 2$ then $\Gr(2, F/H) \cong X$. 

It was overlooked in \cite{CH3} that the statement (2) is meaningless when $\rank(F/H) = 1$. But it is not difficult to show that this case does not arise for a general  $F$. The corresponding case for orthogonal bundles of odd rank will be discussed in the last part of the proof of Proposition \ref{finite}. \qed

\end{remark}

\subsection{Upper bound on the invariant $t(V)$}

\begin{prop} \label{evenbound} {\rm(}\cite[Theorem 1.3]{CH3}{\rm)} Any orthogonal bundle $V$ of rank $2n+2 \ge 4$ satisfies $t(V) \le (n+1)(g-1)+3$. \qed
\end{prop}

\section{Orthogonal extensions} \label{extlift}

\subsection{Orthogonal bundles of odd rank as iterated extensions}
The goal of this section is to to find a criterion similar to that in Proposition \ref{evenrankcrit} for constructing orthogonal bundles of odd rank as extensions. 

Let $E$ be any subbundle of an orthogonal bundle $V$. Then we obtain the sequence
\begin{equation} \label{basicext}
0 \to E \to V \to (E^\perp)^* \to 0.
\end{equation}
The isotropy of $E$ is by definition equivalent to $E \subseteq E^{\perp}$. 

\begin{lem} \label{Morth} Suppose $E$ is a Lagrangian subbundle of $V$. Then the quotient $E^\perp /E$ is isomorphic to $\ox$.
\end{lem}
\begin{proof}
From the sequence (\ref{basicext}) with $\det V = \ox$, we see that $\det E^\perp = \det E$. Since $E^\perp / E$ is a line bundle, it must be isomorphic to $\ox$. 
\end{proof}
Henceforth, we consider an arbitrary extension of vector bundles
\begin{equation} \label{extF}
0 \to E \stackrel{j}{\to} F \to \ox \to 0.
\end{equation}
For any extension $V$ of $F^*$ by $E$, consider the diagram
\begin{equation} \label{dual}
\begin{CD}
@. @. @. \ox \\
@. @. @. @VVV \\
0 @>>>  E @>>> V @>>> F^*  @>>> 0 \\
 @. @VjVV @.  @V^t jVV \\
0 @>>> F @>>> V^* @>>> E^* @>>> 0\\
@. @VVV\\
@. \ox
\end{CD}
\end{equation}

If $V$ has an orthogonal structure with respect to which $E$ is Lagrangian, then $F \cong E^\perp$, and the induced isomorphism $\omega \colon V \to V^*$ fits into the above diagram to yield commutative squares. Furthermore, the induced map $\bar{\omega} \colon \ox \to \ox$ described by the Snake Lemma is a (symmetric) isomorphism.
\begin{prop} \label{cohomcrit}
Let $V$ be a vector bundle fitting into the above diagram (\ref{dual}), with extension class $[V] \in H^1(X, F \otimes E)$. Then there is a symmetric map $\omega \colon V \to V^*$ (not necessarily an isomorphism) extending $j$ and $^t j$, if and only if the image of $[V]$ under the map $j_* \colon H^1(X, F \otimes E) \longrightarrow H^1(X, \otimes^2 F)$ lies on the subspace  $H^1(X, \wedge^2 F)$.
\end{prop}
\begin{proof}
By a standard cohomological argument, there is a map $\omega$ extending $j$ and $^t j$ if and only if
\begin{equation} j_* [V] \ = \ ^t j^* [V^* ] \ \hbox{in} \ H^1 (X, \otimes^2 F). \label{twoone} \end{equation}
It is well known that $[ V^* ] = - \, ^t[V]$ in $H^1 (X, E \otimes F)$. Hence
\[ ^t j^* [ V^* ] \ = \ ^t j^* \left( -\,^t[V] \right) \ = \ -\, ^t\left( j_* [V] \right). \]
Therefore the condition (\ref{twoone}) becomes $j_* [V]  =  - \,^t\left( j_* [V] \right)$, that is, $j_* [V] \in H^1 (\wedge^2 F)$.
It remains to show that such an $\omega$ may be assumed to be symmetric. If $\omega$ exists, it is given with respect to a suitable open covering $\{ U_i \}$ by a cochain of matrices of the form
\[ \omega_i = \begin{pmatrix} 0 & 0 & I_n \\ 0 & \lambda & \star \\ I_n & \star & \star \end{pmatrix}, \]
where $\lambda$ is a scalar (possibly zero). If $U_i$ and $U_j$ are open sets in the covering and $f_{ij}$ is the transition function for $V$ over $U_i \cap U_j$, then the cochain $\{ \omega_i \}$ satisfies the cocycle condition
\begin{equation} ^t f_{ij}^{-1} \omega_j \ = \ \omega_i f_{ij} \label{cocycle} \end{equation}
Taking transposes, we obtain $^t \omega_{j} f_{ij}^{-1} \ = \ ^tf_{ij} \, ^t\omega_i$, whence
\begin{equation} ^t f_{ij}^{-1} \, ^t\omega_j \ = \ ^t \omega_i f_{ij}. \label{cocycletranspose} \end{equation}
Adding (\ref{cocycle}) and (\ref{cocycletranspose}), we see that $\frac{1}{2} ( ^t \omega + \omega )$ is a global section of $\Sym^2 V^*$. This is clearly a symmetric map extending $j$ and $^t j$, as required. \end{proof}

It is straightforward to check that
\[ (\wedge^2 F) \ \cap \ ( F \otimes E ) \ = \ \wedge^2 E \quad \hbox{and} \quad \frac{\wedge^2 F}{\wedge^2 E} \ \cong \ E, \]
the latter holding because $\rank E = \rank F - 1$. We obtain the associated diagram of cohomology groups:
\begin{equation}  \label{twotwo}
\begin{CD} 
H^0(X, \ox)  @>\partial>> H^1( X, E ) @>>> H^1(X, F) @>>> \cdots \\
@. @A{\rho}AA @AAA  @. \\
0 @>>> H^1(X, \wedge^2 F) @>>> H^1(X, \otimes^2 F) @>>> \cdots \\
@. @AAA @Aj_*AA  @. \\
\cdots @>>> H^1(X, \wedge^2 E) @>>> H^1(X, F \otimes E)  @>>> \cdots 
  \end{CD}
\end{equation}
\begin{lem} \label{h1inj}
Suppose that $E$ is a stable bundle of negative degree. Then $j_*$ and the two maps from $H^1(X, \wedge^2 E)$ are injective.
\end{lem}
\begin{proof}
As $E$ is stable of negative degree, $H^0(X, E) = 0$ and the map $H^1(X, \wedge^2 E) \to H^1(X, \wedge^2 F)$ is injective. Since the composition
\[ H^1(X, \wedge^2 E) \to H^1(X, \wedge^2 F) \to H^1(X, \otimes^2 F) \]
is injective, so is the map $H^1(X, \wedge^2 E) \to H^1(X, F \otimes E)$. Finally, note that $\partial$ is the map whose image corresponds to the extension class (\ref{extF}) of $F$ in $H^1 (X, E)$. Therefore $\partial$ is nonzero, hence injective, and so $H^0(X, F) \cong H^0(X, E) = 0$. This shows that $j_*$ is injective.
\end{proof}
We write $\cc_j$ for the image of $H^0( X, \ox ) \cong \cc $ inside $H^1(X, E)$, because it corresponds to the class of the extension $0 \to E \stackrel{j}{\to} F \to \ox \to 0$ in $H^1 (X, E)$, and the above diagram arises from ${j}$. 
\begin{prop} Inside the extension space $H^1(X, F \otimes E)$, the locus of extensions $[V]$ admitting a symmetric map $\omega \colon V \to V^*$ extending $j$ and $^t j$ in the diagram (\ref{dual}) is identified with the preimage $\rho^{-1} \left( \cc_j \right)$.
\end{prop}
\begin{proof} By Proposition \ref{cohomcrit}, this locus is identified with the intersection
\[ j_* \left( H^1(X, F \otimes E) \right) \ \cap \ H^1(X, \wedge^2 F). \]
This coincides with the kernel of the composed map $H^1 (X, \wedge^2 F) \to H^1 (X, \otimes^2 F) \to H^1 (X, F)$. By commutativity of (\ref{twotwo}), this kernel is precisely $\rho^{-1} \left( \cc_j \right)$. \end{proof}

Henceforth we write $\Pi_j := \rho^{-1} ( \cc_j )$. We have an exact sequence of vector spaces
\begin{equation} \label{twofour} 0 \to H^1 ( X, \wedge^2 E ) \to \Pi_j \to \cc_j \to 0. \end{equation}

Thus we have a criterion for the existence of a symmetric map $\omega \colon V \to V^*$. Now we want to know when such a map is an isomorphism.

\begin{lem} \label{notisom} Let $E$ be a stable bundle of negative degree, and consider an extension class $[V] \in \Pi_j$. 
Then the following statements are equivalent:
\begin{enumerate}
\renewcommand{\labelenumi}{(\arabic{enumi})}
\item The extension class $[V]$ belongs to $H^1 (X, \wedge^2 E)$.
\item The map $\ox \to F^*$ lifts to $V$. In other words, there is a map $\ox \to V$ whose composition with $V \to F^*$ coincides with the given map $\ox \to F^*$. 
\item The map $\omega \colon V \to V^*$ is not an isomorphism.
\end{enumerate} \end{lem}

\begin{proof}
$(1) \Leftrightarrow (2)$: Firstly, consider the exact commmutative diagram
\[\begin{CD}
H^0(X, F)  @>>> H^1(X, E) @>>> H^1 \left(X, \frac{\otimes^2 F}{\otimes^2 E} \right) @>>> H^1(X, F)  \\
@. @A{\rho}AA @AAA @AAA @. \\
@. H^1(X, \wedge^2 F) @>>> H^1(X, \otimes^2 F) @>>> H^1(X, \Sym^2 F)  \\
@. @AAA @AAA @AAA @. \\
  @. H^1(X, \wedge^2 E) @>>> H^1(X, \otimes^2 E)  @>>> H^1 \left(X, \Sym^2 E \right)
  \end{CD}\]
As was seen in the proof of Lemma \ref{h1inj}, we have $h^0 (X, F) = 0$. It follows that
\begin{equation} H^1 (X, \otimes^2 E) \cap H^1 (X,  \wedge^2 F ) \ = \ H^1 (X, \wedge^2 E). \label{wedgeint} \end{equation}

Now the map $\ox \to F^*$ lifts to $V$ if and only if $[V]$ belongs to 
\[ H^1 (X, \otimes^2 E) \ = \ \Ker \left( H^1 (X, F \otimes E) \to H^1 (X, E ) \right). \]
Since $[V]$ belongs to $\Pi_j \subset H^1(X, \wedge^2 F)$, this condition is equivalent to $[V] \in H^1(X, \wedge^2 E)$ in view of (\ref{wedgeint}).

$(2) \Leftrightarrow (3)$: By the Snake Lemma, we know $\Ker \, \omega$ coincides with the kernel of the induced map $\bar{\omega} \colon \ox \to \ox$ in (\ref{dual}). Thus $\omega$ fails to be an isomorphism if and only if $\bar{\omega}$ is zero, which is equivalent to the lifting of $\ox \to F^*$ to $V$. \end{proof}

\begin{remark} \label{degenerate} The bilinear form on any $V$ satisfying the equivalent conditions of the lemma is degenerate. The class $[V]$ is of the form $^t j^*\left[ V_0 \right]$, where $[V_0] \in H^1(X, \wedge^2 E)$ defines an orthogonal extension in the sense of Proposition \ref{evenrankcrit}. Hence $V$ is an extension $0 \to \ox \to V \to V_0 \to 0$. However, if $V_0$ is semistable, then $V$ is S-equivalent as a vector bundle to the orthogonal bundle $V_0 \perp \ox$, which admits a nondegenerate symmetric form. If $V_0$ is stable as a vector bundle, then $V_0 \perp \ox$ is a stable orthogonal bundle by Ramanan \cite{Ram}. \qed 
\end{remark}

\noindent From the discussion in this section, we can conclude:

\begin{prop} \label{oddrankcrit}  \ 
Let $E$ be a stable bundle of rank $n$ and negative degree. Let $F$ be an extension of $\ox$ by $E$.
\begin{enumerate}
\renewcommand{\labelenumi}{(\arabic{enumi})}
\item An extension $0 \to E \to V \to F^* \to 0$ is induced by an orthogonal structure on $V$ with respect to which $E$ is Lagrangian, if and only if $[V]$ belongs to the complement of $j_* \left( H^1 (X, \wedge^2 E) \right)$ inside $\Pi_j$.
\item If $[ V' ]$ in $H^1 (X, \wedge^2 E)$ defines a semistable bundle $V'$, then the vector bundle defined by $j_* [V']$ is S-equivalent to a semistable orthogonal bundle. \qed \end{enumerate}
\end{prop}

\noindent In \S \ref{stability} we will discuss stability of such extensions.

\subsection{Alternative description of orthogonal extensions}

In the previous subsection, orthogonal extensions were described as iterated extensions: start from a stable bundle $E$ of rank $n$ and degree $<0$, choose an extension $j \in H^1(X, E)$ corresponding to a bundle $F$ of rank $n+1$, and then choose a class in $\Pi_j \subset H^1(X, \wedge^2 F)$. There is another way to describe the same extension, starting from $F$. In this subsection, we explain this alternative, which will be useful later. 

For each positive integer $e$, let $S_e'$ be the Brill-Noether locus 
defined by 
\[
S_e' \ = \ \{ F \in U_X(n+1, -e) : h^0(X, F^*) >0 \}.
\]
The following result is due to Sundaram \cite{Sun}, with notations adapted to our situation:
\begin{prop} Suppose $0 < e \le (n+1)(g-1)$.
\begin{enumerate}
\renewcommand{\labelenumi}{(\arabic{enumi})}
\item The locus $S_e'$ is an irreducible variety of codimension $(n+1)(g-1) -e +1$.
\item For a general $F \in S_e'$ we have $h^0(X, F^*)=1$.  \qed
\end{enumerate}
\end{prop}

We write $S_e$ for the dense subset of $S_e'$ with $h^0(X, F^*)=1$. Note that any $F \in S_e$ admits a distinguished subbundle $E = \Ker[F \to \ox]$.

\begin{lem} \label{Fstable} \ For $0 < e \le \frac{1}{2}(n+1)(g-1)$, a general extension $F \in H^1(X, E)$ of $\ox$ by a general $E \in U_X(n, -e)$ is stable and lies inside $S_e$. \end{lem}
\begin{proof}
Firstly we show that such an $F$ is stable. Suppose there is a destabilizing subbundle $G$ of $F$. Let $K = \Ker[G \to F \to \ox]$ and $Q = G/K$. Now $Q$ must be nonzero, for by stability of $E$ we have
\[ \mu(K) \ < \ \mu(E) \ = \ -\frac{e}{n} \ < \ -\frac{e}{n+1} \ = \ \mu (F). \]
If $K=0$, then $G = \ox(-D)$ for some effective divisor $D$ of degree $d$, and by Proposition \ref{linelift} the extension class $[F]$ lies on $\Sec^d \pp E$.  Since $G$ destabilizes $F$, we have $(n+1)d \le e$. But 
 \[
 \dim \Sec^d \pp E \ \le \ (n+1)d -1 \ \le \ e -1 \ < \ e + n(g-1) - 1 \ = \ \dim \pp H^1(X, E).
 \]
This shows that a general extension $F$ represented in $H^1(X, E)$ does not admit a destabilizing line subbundle. 

Now suppose both $K$ and $Q$ are nonzero. Write $r := \rank K$. We have $r < n$, and in particular $n \geq 2$. Since $E$ is general, a dimension count (see for example Hirschowitz \cite[Th\'eor\`eme 4.4]{Hir}) shows that
\begin{equation} \label{degI}
-re -n \cdot \deg K \ \ge \ r(n-r)(g-1).
\end{equation}
Since $e \le \frac{1}{2}(n+1)(g-1)$ and $r \ge 1$, we have also
 \[
 \frac{n-r}{n+1} e \ < \ r(n-r)(g-1).
 \]
Combining this with (\ref{degI}), we obtain
 \[
 n \cdot \deg (K) \ \le \ -re - r(n-r)(g-1) \ < \ -re - \frac{n-r}{n+1} e \ = \ - \frac{(r+1)n}{n+1} e,
 \]
and so $\deg(K) < -\frac{(r+1)}{n+1} e$. Therefore,
\[
\mu(G) \ = \ \frac{\deg(K) + \deg(Q)}{r+1} \ \le \ \frac{\deg K}{r+1} \ < \ \frac{-e}{n+1} \ = \ \mu(F).
\]

Thus a general extension $0\to E \to F \to \ox \to 0$ lies in $S_e'$. To see that in fact it belongs to $S_e$, consider the split extension $E^* \oplus \ox$ for a general $E$. Here
\[ \mu \left( E^* \right) \ = \ \frac{e}{n} \ \le \ \frac{n+1}{2n} (g-1) \ \le \ g-1. \]
For general $E$, therefore, $h^0(X, E^*) = 0$ and $h^0(X, F^*) = 1$. This implies that $h^0(X, F^*) = 1$ for a general extension $F$ as above. Therefore, we conclude that $F \in S_e$. 
 \end{proof}

\subsection{Interplay between even rank and odd rank}

Here we exploit and generalize some well-known facts on orthogonal Grassmannians. Let $\OG(n, 2n+1)$ (resp., $\OG(n, 2n)$) be the \textit{odd orthogonal Grassmannian} (resp., \textit{even orthogonal Grassmannian}) parameterizing Lagrangian subspaces of $\cc^{2n+1}$ (resp., $\cc^{2n}$). 
\begin{lem} \label{OG} \begin{enumerate}
\renewcommand{\labelenumi}{(\arabic{enumi})}
\item The odd orthogonal Grassmannian $\OG(n, 2n+1)$ is irreducible.
\item The even orthogonal Grassmannian $\OG(n, 2n)$ consists of two disjoint irreducible components $\OG(n, 2n)_1$ and $\OG(n, 2n)_2$ of the same dimension. Two Lagrangian subspaces $\mathsf{F}$ and $\mathsf{F}'$ belong to the same component if and only if $\dim( \mathsf{F} \cap \mathsf{F}') \equiv n \mod 2$.
\item There is a canonical isomorphism $\OG(n, 2n+1) \xrightarrow{\sim} \OG(n+1, 2n+2)_i$ for $i = 1, 2$.
\item For $\mathsf{E} \in \OG(n, 2n+1)$, let $\mathsf{F}_i \in \OG(n+1, 2n+2)_i$ be the Lagrangian subspace obtained via the above isomorphism.  Then 
\[
T_{\mathsf{E}} \OG(n, 2n+1) \ \cong \ T_{\mathsf{F}_i} \OG(n+1, 2n+2)_i \ \cong \ \wedge^2 \mathsf{F}_i^*.
\]
\end{enumerate}
\end{lem}
\begin{proof}
The proofs of (1) and (2) can be found in Reid \cite[\S 1]{Reid}.

To prove (3), let $\mathsf{W}$ be an orthogonal vector space of dimension $2n+2$ and $\mathsf{V}$ an orthogonal subspace of dimension $2n+1$. Given $\mathsf{F} \in \OG(n+1, \mathsf{W})_i$, write $\mathsf{E} := \mathsf{F} \cap \mathsf{V}$. Then $\mathsf{E}$ is isotropic of dimension $n$, since the dimension of an isotropic subspace cannot exceed $\frac{1}{2}\dim \mathsf{V}$. It is easy to see that the assignment $\mathsf{F} \mapsto \mathsf{E}$ yields a surjection $\OG(n+1, \mathsf{W})_i \to \OG(n, \mathsf{V})$. To show the injectivity, let $\mathsf{F}$ and $\mathsf{F}'$ be two Lagrangian subspaces of $\mathsf{W}$ in the same component, both containing $\mathsf{E}$. By (2), we have
\[ n \ \le \ \dim \left( \mathsf{F} \cap \mathsf{F}' \right) \ \equiv \ n+1 \mod 2, \]
so $\mathsf{F} = \mathsf{F}'$.  

The first isomorphism of (4) is immediate from (3). The second isomorphism of (4) is a well-known fact on even orthogonal Grassmannians.
\end{proof}

Now let $0 \to E \to V \to F^* \to 0$ be an orthogonal extension of rank $2n+1$ as in Proposition \ref{oddrankcrit}. By Proposition \ref{cohomcrit}, the class $j_* [V]$ belongs to $H^1 (X, \wedge^2 F)$. In view of Proposition \ref{evenrankcrit}, we obtain an orthogonal extension $0 \to F \to W \to F^* \to 0$ with $[W] = j_* [V]$. Note that the $F$ in this sequence is a new copy of $F$ distinct from $E^\perp \subset V$. Henceforth, we denote it $F'$ for clarity. We obtain an exact diagram:
\begin{equation} \label{WV} \begin{CD} 
 @. 0 @. 0 @.  @. \\
 @. @VVV @VVV @. @. \\
 0 @>>> E @>>> V @>>> F^* @>>> 0 \\
 @. @VVV @VVV @| @. \\
 0 @>>> F' @>>> W @>>> F^* @>>> 0 \\
 @. @VVV  @VVV  @. @. \\
 @. \ox @= \ox @. @. @. \\
 @. @VVV @VVV @. @. \\
 @. 0 @. 0 @. @. \end{CD} \end{equation}
\begin{prop} \begin{enumerate}
\renewcommand{\labelenumi}{(\arabic{enumi})}
\item The bundle $W$ is an orthogonal direct sum $V \perp \ox$. The subbundle $F'$ is Lagrangian in $W$ and satisfies $F' \cap V = E$.
\item For every Lagrangian subbundle $G$ of $W$, the intersection $G \cap V$ is a Lagrangian subbundle of $V$ with $\det (G \cap V) = \det G$.
\item The association $G \mapsto G \cap V$ defines a surjective $2:1$ map
\[ \{ \text{Lagrangian subbundles of $W$} \} \leftrightarrow \{ \text{Lagrangian subbundles of $V$} \}. \]
\end{enumerate}
\label{Lagsubbs} \end{prop}
\begin{proof}
(1) The only statement not clear from the diagram and the discussion before it is that $W = V \perp \ox$. It is easy to see that the orthogonal complement $V^\perp \subset W$ is isomorphic to $\ox$. 
Since the form on $V$ is nondegenerate, $V \cap V^\perp$ is everywhere zero. Hence $V^\perp$ is not isotropic, and is therefore isomorphic to $\ox$ as an orthogonal bundle. From the fact that $V \cap V^\perp = 0$ it also follows that $V^{\perp} \hookrightarrow W$ is a splitting of $W \to \ox$.

Statement (2) follows from the linear algebraic fact in Lemma \ref{OG} (3) and the diagram
\[ \begin{CD} 0 @>>> V @>>> W @>>> \ox @>>> 0 \\
 @. @AAA  @AAA  @| @. \\
 0 @>>> G \cap V @>>> G @>>> \ox @>>> 0. \end{CD}  \]

As for (3): Suppose $E$ is a Lagrangian subbundle of $V$. Consider a trivalization of $W$ over an open subset $U \subseteq X$. By Lemma \ref{OG} (3), the bundle $E|_U$ can be completed to a Lagrangian subbundle of $W|_U$ in exactly two ways. Since $\dim X = 1$, each such completion admits a unique extension to the whole of $X$. Therefore, the map $G \mapsto G \cap V$ is two-to-one.
\end{proof}

\begin{remark} Given $E \xrightarrow{j} E^\perp = F \subset V$, the Lagrangian subbundles $G \subset W$ satisfying $G \cap V = E$ may be realized as follows: One such $G$ comes from the construction in Proposition \ref{Lagsubbs} (1). The other is the image of the first via the orthogonal automorphism of $W = V \perp \ox$ given by $(v, \lambda) \mapsto (v, -\lambda)$. \end{remark}


We can now deduce a topological characterization of the two components of $MO_X (2n + 1)$, from the analogous one stated in Proposition \ref{parityeven} for bundles of even rank:

\begin{theorem} \label{thm1} \begin{enumerate}
\renewcommand{\labelenumi}{(\arabic{enumi})}
\item Let $E_1$ and $E_1$ be Lagrangian subbundles of an orthogonal bundle $V \in MO_X (2n+1)$. Then $\deg E_1$ and $\deg E_2$ have the same parity.
\item The bundle $V$ belongs to $MO_X(2n+1)^+$ (resp., $MO_X(2n+1)^-$) if and only if its Lagrangian subbundles have even degree (resp., odd degree).
\end{enumerate}
\end{theorem}

\begin{proof}
This follows from Proposition \ref{parityeven} and Proposition \ref{Lagsubbs} (2) and (3).
\end{proof}

\noindent We also obtain an upper bound on $t(V)$:

\begin{prop}  \label{oddbound}
For any orthogonal bundle $V$ of rank $2n+1$, we have $t(V) \le (n+1)(g-1)+3$. 
\end{prop}
\begin{proof}
Let $W = V \perp \ox$. By Proposition \ref{evenbound}, the bundle $W$ has a Lagrangian subbundle $F$ with $-2\deg F \le (n+1)(g-1)+3$. By Proposition \ref{Lagsubbs}, the bundle $E = F \cap V$ is a Lagrangian subbundle of $V$ with $\deg E = \deg F$. Hence we obtain the desired bound on $t(V)$.
\end{proof}

\noindent In \S \ref{Segrestrat}, we will show that this upper bound is sharp.

\section{Parameter spaces of orthogonal extensions}

\subsection{Construction of the parameter space} \ Firstly, we construct parameter spaces for certain orthogonal extensions of the form
\begin{equation} 0 \to E \to V \to F^* \to 0 \label{stdext} \end{equation}
where $E$ is a bundle of rank $n$ and degree $-e < 0$ and $F$ is an extension $0 \to E \xrightarrow{j} F \to \ox \to 0$. The parameter space will be obtained by deforming the space $\pp \Pi_j$ obtained in the previous section in a suitable way.

Let $U_X^s(n, -e)$ be the moduli space of stable bundles of rank $n$ and degree $-e$. This is a quasiprojective irreducible variety of dimension $n^2(g-1) + 1$.

\begin{prop} \label{taut} For each positive integer $e$, we have:
\begin{enumerate}
\renewcommand{\labelenumi}{(\arabic{enumi})} 
\item There exists a finite \'etale cover $\xi_e \colon \tsu \to U_X^s(n, -e)$ and a double fibration
\[
\A_e \ \stackrel{\pi_e}{\longrightarrow} \ J_e \ \stackrel{\tau_e}{\longrightarrow} \ \tsu
\]
such that the fiber $\tau^{-1} (\bar{E})$ with $\xi_e(\bar{E}) = E$ is isomorphic to the (projectivized) extension space $\pp H^1(X, E)$, and the fiber of $\pi_e$ at $j \in \pp H^1(X, E)$ is isomorphic to $\pp \Pi_j$.

\item There is a bundle $\V_e$ over $\A_e \times X$ such that for each $j \in \pp H^1(X, E)$ and $[V] \in \pp \Pi_j$, the restriction of $\V_e$ to $\{ [V] \} \times X$ is isomorphic to the orthogonal bundle $V$.
\item The variety $\A_e$ has dimension $\frac{1}{2} n(3n+1) (g-1) + ne$. 
\end{enumerate} 
\end{prop}

\begin{proof} (1) We follow the construction in \cite[\S 3.1]{CH3}. By Narasimhan--Ramanan \cite[Proposition 2.4]{NR}, there exists a finite \'etale cover $\xi_e \colon \tsu \to U_X^s(n, -e)$ together with a bundle $\mathcal E \to \tsu \times X$, such that the restriction of $\mathcal E$ to $\{ \bar{E} \} \times X$ is isomorphic to the bundle $E = \xi_e(\bar{E})$. By a standard process with universal extensions, we find a fibration $\tau_e \colon J_e \to \tsu$ with fiber $\tau_e^{-1} (\bar{E}) = \pp H^1(X, E)$, where $E = \xi_e (\bar{E})$. Furthermore, write $p_J \colon J_e \times X \to J_e$ and $p_X \colon J_e \times X \to X$ for the projections. Over $J_e \times X$ there is an exact sequence of bundles
\[ 0 \to \mathcal E \to \mathcal F \to p_X^* \ox \to 0, \]
whose restriction to $\{ {j} \} \times X$ is the extension $0 \to E \stackrel{j}{\to} F \to \ox \to 0$.

Globalizing (\ref{twotwo}), we consider the sheaf
\[ \Ker \left[ R^1 {\left( p_{J} \right)_*}  (\wedge^2 \mathcal F ) \ \longrightarrow \ R^1{\left( p_{J} \right)_*}  (\otimes^2 \mathcal F ) \ \longrightarrow \ R^1 {\left( p_{J} \right)_*} (\mathcal F \otimes p_X^* \ox) \right]. \]
By Lemma \ref{h1inj}, the restriction of this kernel to each $j \in J_e$ is of constant dimension. Hence we obtain a vector bundle over $J_e$ whose fiber at $j$ is $\Pi_j$. We denote the corresponding projective bundle over $J_e$ by $\A_e$.\\
\\
(2) The existence of the universal extension $\V_e$ over $\tse \times X$ follows from Lange \cite[Corollary 4.5]{Lange}.\\
\\
(3) By Lemma \ref{h1inj}, we have $\dim \pp \Pi_j = h^1 (X, \wedge^2 E)$. Since $E$ is stable of negative degree, $h^1(X, E) \ = \ e + n(g-1)$ and $h^1 (X, \wedge^2 E) \ = \ (n-1)e + \frac{1}{2}n(n-1)(g-1)$.
Thus $\dim \A_e$ is given by
\[ \dim \tsu + \dim \pp H^1(X, E) + \dim \pp \Pi_j \ = \ \frac{1}{2} n(3n+1) (g-1) + ne. \]
\end{proof}

\subsection{Rank 3 case} \label{rkthree} \ Consider an orthogonal bundle $V$ of rank $3$. In this case, $E$ is a line bundle and $\wedge^2 E = 0$. Therefore, $\Pi_j \cong \cc$, corresponding to the extension class $j \in H^1(X, E)$. In other words, for each rank two extension $0 \to E \stackrel{j}{\to} F \to \ox \to 0$, there is, up to isomorphism, a unique orthogonal bundle $V$ containing a Lagrangian subbundle $E$ with $F = E^\perp$. The parameter space $\A_e$ coincides with $J_e$, which admits a fibration $\tau_e \colon J_e \to \Pic^{-e}(X)$ with fiber $\tau_e^{-1} (E) = \pp H^1(X, E)$. 
On the other hand, $\A_e$ is birational to $S_e \subset U_X(2, -e)$ by Lemma \ref{Fstable}. Note that $\dim \A_e = e + 2g-2$. From Mumford \cite[p.\ 185]{M} we recall the following statement:
\begin{lem} Every orthogonal bundle $V$ of rank 3 is of the form $L \otimes S^2 F$, where $L$ is a line bundle and $F$ is a rank 2 bundle with $\det F \cong L^*$. 
\end{lem}
\begin{proof}
The bundle $V := L \otimes S^2F $ is orthogonal, due to the symmetric isomorphism 
\[
V^* \cong L^* \otimes S^2(F^*) \cong L^* \otimes S^2(F \otimes L) \cong L \otimes S^2 F  = V.
\]
Let $V$ be an orthogonal bundle of rank 3 with a Lagrangian subbundle $E$ and write $F := E^\perp$. Tensoring the exact sequence
\[
0 \to E \otimes F \to S^2 F \to \ox \to 0,
\]
by $E^*$, we obtain $0 \to F \to E^* \otimes S^2 F \to E^* \to 0$. By the uniqueness discussed before the statement of the lemma, $V$ is isomorphic to $E^* \otimes S^2 F$. 
\end{proof}

\subsection{Stability of general bundles} \label{stability} \ The space $\A_e$ parameterizes those orthogonal bundles  of rank $2n+1$ which admit Lagrangian subbundles of degree $-e $. By the universal property, there is a rational map $\alpha_e \colon \A_e \dashrightarrow MO_X(2n+1)$. Our next step will be to show that a general point of $\A_e$ corresponds to a stable orthogonal bundle. This will imply that the maps $\alpha_e$ are defined on dense subsets.

\begin{prop} \label{generalstable}
For $0 < e < \frac{1}{2}(n+1) (g-1)$, a general bundle represented in $\A_e$ is a stable orthogonal bundle. 
\end{prop}
\begin{proof}  
Firstly, we consider the case $n=1$. Suppose $V_0$ is an orthogonal bundle of rank 3 and $G \subset V_0$ an isotropic line subbundle of nonnegative degree. From the sequence $0 \to E \to V_0 \to F^* \to 0$ and its dual sequence, 
we see that $G$ is a subsheaf of both $F^*$ and $E^*$. From the sequence
\[
0 \to \ox \to F^* \xrightarrow{^t j} E^* \to 0,
\]
it follows that $G \cong E^*(-D)$ for some $D$ of degree $\le e$, and that $G$ lifts to a subsheaf of $F^*$. By Proposition \ref{linelift}, we have $[^t j] \in \Sec^e X$ for the embedded curve $X \subset \pp H^1(X, E)$. But since $e < g-1$ by hypothesis,
\[
\dim \Sec^e X \ \le \ 2e-1 \ < \ e+g-2 \ = \ \dim \pp H^1(X, E).
\]
Thus a general extension class $[^t j]$ yields a stable orthogonal bundle in $\Pi_j$.

Now suppose $n \ge 2$. Let $E \in U_X^s(n, -e)$ be general, and let $0 \to E \xrightarrow{j} F \to \ox \to 0$ be an extension. Let $0 \to E \to \tilde{V}_0 \to E^* \to 0$ be an extension whose class $[\tilde{V}_0]$ is a general point of $H^1(X, \wedge^2 E)$. By Proposition \ref{evenrankcrit}, the bundle $V_0$ admits an orthogonal structure.

Consider now the orthogonal extension $0 \to E \to V_0 \to F^* \to 0$ defined by $[V_0] \ = \ ^tj^*[\tilde{V}_0]$. We obtain a diagram
\[ \begin{CD} 0 @>>> E @>>> \tilde{V}_0 @>>> E^* @>>> 0 \\
 @. @| @AAA @AA^tjA @. \\
 0 @>>> E @>>> V_0 @>>> F^* @>>> 0 \\
 @. @. @AAA @AAA @. \\
 @. @. \ox @= \ox @. \end{CD} \]
By Lemma \ref{notisom} and Remark \ref{degenerate}, there is an exact sequence
\[ 0 \to \ox \to V_0 \xrightarrow{\omega} V_0^* \to \ox \to 0 \]
where $\omega$ defines a degenerate symmetric form on $V_0$ (the pullback of the form on $\tilde{V}_0$). By Lemma \ref{notisom}, however, a generic deformation of $V_0$ in $\Pi_j$ admits an orthogonal structure.

Suppose there is an isotropic subbundle $G$ of $V_0$ of nonnegative degree. Then we have a diagram
\[ \begin{CD} 0 @>>> \ox @>>> V_0 @>>> \tilde{V}_0 @>>> 0 \\
 @. @AAA @AAA @AAA @. \\
0 @>>> G_1 @>>> G @>>> G_2 @>>> 0 \end{CD} \]
where $G_1$ is either zero or $\ox$.  
Since $[ \tilde{V}_0 ]$ is general, $\tilde{V}_0$ is a stable orthogonal bundle by \cite[Theorem 3.4]{CH3}. If $G_2$ is nonzero, then it is isotropic in $\tilde{V}_0$ since $G$ is isotropic in $V_0$. Hence $\deg G = \deg G_2 < 0$. This shows that the only destabilizing subbundle of $V_0$  is $\ox $. 

Now we deform $V_0$ in $\Pi_j$ to get a family $\{V_\lambda \}$ whose general member is an orthogonal bundle lying on $\A_e$. By semicontinuity, a generic deformation $V_\lambda$ of $V_0$ in $\Pi_j$ can possibly be destabilized only by a line bundle  which is a deformation of $\ox \subset V_0$. Note that this at the same time yields a deformation  of $\ox$ in $F^*$. Such deformations are parameterized by $H^0(X, \Hom(\ox, F^*/\ox)) \cong H^0(X, E^*)$. Since \[
\mu(E^*) \ = \ \frac{e}{n} \ < \ \frac{n+1}{2n} (g-1) \ \le \ g-1,
\]
we have $h^0(X, E^*) = 0$ for a general $E \in U_X^s(n, -e)$. 
 Thus $\ox$ does not deform nontrivially in $F^*$. Since a general deformation $V_t$ does not have a lifting of $\ox \subset F^*$, we conclude that it is a stable orthogonal bundle in $\A_e$. 
\end{proof}

We remark that a general $V \in \A_e$ is stable also for $\frac{1}{2}(n+1)(g-1) \le e \le \frac{1}{2} \left( (n+1)(g-1) +3 \right)$, as will be shown in next section.  We expect the same is true for larger values of $e$, but we do not require this for the present applications.

\section{The Segre stratification} \label{Segrestrat}

Suppose $1 \leq e \leq \frac{1}{2} \left( (n+1)(g-1) + 3 \right)$, and let $E \in U_X^s(n, -e)$ be general. Let $0 \to E \xrightarrow{j} F \to \ox \to 0$ be a general extension. The goal of this section is to show that for a general orthogonal extension $V$ represented in $\Pi_j$, the Lagrangian subbundle $E$ is maximal in $V$. This will confirm that $t(V) = -2 \cdot \deg E$.
 
Consider an orthogonal bundle $V$ of rank $2n+1$ admitting two different Lagrangian subbundles $E$ and $\tilde{E}$, with orthogonal complements $F$ and $\tilde{F}$ respectively. Let $I$ and $H$ be the locally free parts of $E \cap \tilde{E}$ and $F \cap \tilde{F}$ respectively.
\begin{lem} \label{cap}
The subsheaf $I \subset H$ is a subbundle of corank 1.
\end{lem}
\begin{proof}
Since it suffices to give a fiberwise argument, we will regard $E$ and $F$ as vector spaces in the discussion below. We have
\[
F \cap \tF \ = \ E^\perp \cap \tE^\perp \ = \ (E + \tE)^\perp .
\]
Thus if $\dim (E \cap \tE) = r$, then $\dim (E + \tE) = 2n-r$ and $\dim(F \cap \tF) = r+1$.
\end{proof}
\begin{cor} \label{capcor}
In the above context, $H/I \cong \ox(-D)$ for some effective divisor $D$ with $\deg D = \deg I - \deg H$. Furthermore, some extension $H$ of $\ox(-D)$ by $I$ lifts to $F$ if and only if the class $j \in H^1(X, E)$ belongs to
\[
\Ker \left[ H^1(X, E) \to H^1(X, E/I) \to H^1(X, (E/I) (D)) \right].
\]
\end{cor}
\begin{proof}
As a consequence of Lemma \ref{cap}, we obtain the following diagram for a torsion sheaf $\tau_D$ associated to some effective divisor $D$:
\begin{equation} \label{IH}
\begin{CD}  
0 @>>> E/I @>>> F/H @>>> \tau_D @>>> 0  \\
 @. @AAA @AAA  @AAA @. \\
0 @>>> E @>>> F @>>> \ox @>>> 0 \\
 @. @AAA @AAA @AAA @. \\
0 @>>> I @>>> H @>>> H/I @>>> 0 \end{CD} 
\end{equation}
From this it is clear that $H/I$ is of the form $\ox(-D)$ as stated.

For the rest: It is easy to see that some extension $0 \to I \to H \to \ox(-D) \to 0$ lifts to $F$ if and only if $\ox(-D) \subseteq \ox$ lifts to $F/I$. 
This is equivalent to the statement that $j$ belongs to $\Ker \left[ H^1 (X, E) \to H^1 (X, E/I) \to H^1 (X, (E/I)(D)) \right]$.
\end{proof}
Recall now that $H^1(X, \wedge^2 F)$ admits the filtration
\[ H^1(X, \wedge^2 E) \ \subset \ \Pi_j \ \subset \ H^1(X, \wedge^2 F). \]
We write $\widetilde{p_H}$ for the restriction to $\Pi_j$ of the natural surjection $p_H \colon H^1(X, \wedge^2 F) \to H^1(X, \wedge^2 (F/H))$.
\begin{lem} \label{modify}
The map $\widetilde{p_H} \colon \Pi_j \to H^1(X, \wedge^2 (F/H))$ is surjective.
\end{lem}
\begin{proof}
It suffices to show that the restriction of $\widetilde{p_H}$ to the subspace $H^1(X, \wedge^2 E)$ of $\Pi_j$ is surjective. By (\ref{IH}), we have a commutative diagram
\[
\begin{CD}
H^1 (X, \wedge^2 E) @>>> H^1 (X, \wedge^2 F) \\
 @VVV @VVV \\
H^1 (X, \wedge^2(E/I)) @>>> H^1 (X, \wedge^2(F/H)).
\end{CD}
\]
where 
the composition $H^1 (X, \wedge^2 E) \to H^1 (X, \wedge^2 (E/I)) \to H^1 (X, \wedge^2 (F/H))$ is surjective. The statement follows by commutativity of the diagram.
\end{proof}

\noindent Now we will obtain a modification of Proposition \ref{quote} (2). Let $E$ be a general stable bundle of rank $n$ and degree $-e < 0$.  Let $V$ be an orthogonal bundle of rank $2n+1$, admitting $E$ as a Lagrangian subbundle, so that $E^\perp$ is an extension $0 \to E \stackrel{j}{\to} F \to \ox \to 0$.

\begin{lem} \label{liftingcrit} Suppose that $V$ has another Lagrangian subbundle $\tE$ of degree $-\tilde{e}$ with $\tE^\perp = \tF$. As before, write $I$ and $H$ for the locally free parts of $E \cap \tE$ and $F \cap \tF$ respectively, and write $\rank I = r$ and $\deg H = -h$.
\begin{enumerate}
\renewcommand{\labelenumi}{(\arabic{enumi})}
\item The image of the class $j$ under the surjection $H^1(X, E) \to H^1(X, E/I) $ lies inside the affine cone of $\Sec^d \pp (E/I)$, where $d = \deg I + h \ge 0$.
\item Write $k := \frac{1}{2} (e + \tilde{e} - 2h)$. (Note that $e + \tilde{e} \equiv 0 \mod 2$, by Theorem \ref{thm1} (1).) Then $k \ge 0$.
\item If $n \geq 2$ and $r \le n-2$, the image of the class $[V]$ under $\widetilde{p_H} \colon \Pi_j \to H^1(X, \wedge^2 (F/H))$ lies inside the affine cone of $\Sec^k \Gr(2, F/H)$.
\end{enumerate}
In particular, when $E$ and $\tE$ intersect generically in zero, the above statements holds with $\deg I = 0$ and $r = 0$. When $n = 1$, parts (1) and (2) hold with $I = 0$.  
\end{lem}
\begin{proof}
Statement (1) follows from a geometric interpretation of Corollary \ref{capcor} by using Proposition \ref{linelift}.

For the rest: By Proposition \ref{Lagsubbs}, the orthogonal bundle $W = V \perp \ox$ contains two Lagrangian subbundles $F'$ and $\tF'$ isomorphic to $F$ and $\tF$ respectively. We claim that $F' \cap \tF'$ also has locally free part isomorphic to $H$ in $W$. To see this, note that by Lemma \ref{cap} and the diagram (\ref{WV}) there exists a commutative diagram
\[ \xymatrix{ V \ar[rrr] \ar[ddd] \ar[dr] & & & F^* \ar[ddd] \\
 & W \ar[urr] \ar[ddl] \ar[drdr] & & \\
 & & & \\
 \tF^* \ar[rrr] & & & H^* } \]
Dualizing, we see that the Lagrangian subbundles $F'$ and $\tF'$ in $W \cong W^*$ also intersect in a copy of $H$. Statements (2) and (3) now follow from Proposition \ref{quote} (2). \end{proof}


Now we can prove the following statement. 
\begin{prop} \label{finite} \ Assume that $E$ is general in $U_X (n, -e)$ and that $F$ is a general extension of $\ox$ by $E$. For $0 < e < \frac{1}{2}(n+1) (g-1)$, a general orthogonal extension $0 \to E \to V \to F^* \to 0$ has no Lagrangian subbundle of degree $\ge -e$ other than $E$. Therefore, $t(V) = 2e$ and $E$ is the unique maximal Lagrangian subbundle of $V$. 
\end{prop}
\begin{proof}
We will prove the proposition in the following way: Recall the double fibration $\A_e \xrightarrow{\pi_e} J_e \xrightarrow{\tau} \tsu$ described in Proposition \ref{taut}. By Proposition \ref{oddrankcrit}, for fixed $\bar{E} \in \tsu$ mapping to $E \in U_X (n, -e)$, the fiber $\left( \tau_e \circ \pi_e \right)^{-1}(\bar{E})$ parameterizes all orthogonal bundles $V$ containing $E$ as a Lagrangian subbundle. We will show that for general $E$, the locus of such $V$ containing a Lagrangian subbundle of degree $\ge -e$ apart from the original $E$ has positive codimension in $\left( \tau_e \circ \pi_e \right)^{-1}(\bar{E})$. 

Assume, then, that an orthogonal extension $0 \to E \to V \to F^* \to 0$ has another Lagrangian subbundle $\tilde{E}$ of degree $-\tilde{e} \ge -e$. Let $\tF, I$ and $H$ be as in Lemma \ref{liftingcrit}, with $\rank I = r$ and $\deg H = -h$. 

Firstly, suppose $n = 1$, so $r=0$. By (\ref{IH}), we see that $H = \ox(-D)$ and $d = h$. By Lemma \ref{liftingcrit} (1), the class $j \in \pp H^1(X, E)$ lies inside $\Sec^h X$, where $h \le \frac{1}{2}(e + \tilde{e}) \le e$. As $e < g-1$, we have
\[
\dim \Sec^h X \ \le \ 2e-1\ < \ e+g-2 \ = \ \dim \pp H^1(X, E).
\]
Thus a general $j$ lies outside $\Sec^h X$ in $H^1 (X, E)$, and the unique orthogonal extension $0 \to E \to V \to F^* \to 0$ discussed in \S \ref{rkthree} has the unique maximal Lagrangian subbundle $E$.

Now suppose $n \ge 2$ and $0 \le r \le n-2$. 
To bound the dimension of those $V$ containing such an $\tE$, we need to compute four dimensions. Firstly, by Lemma \ref{liftingcrit} (1), the class $j$ varies inside an algebraic subset of dimension bounded by 
\[
D_1 :=  \dim \Sec^d \pp (E/I) + 1 + \dim \Ker \left[ H^1 (X, E) \to H^1 (X, E/I) \right].
\]
We have $\dim \Sec^d \pp (E/I) \ \le \ d(n-r-1) - 1 \ = \ (\deg I + h)(n-r-1) - 1$. 
Moreover, the dimension of the fiber of $H^1 (X, E) \to H^1(X, E/I)$ is bounded by $h^1(X, I) = h + r(g-1)$, since $h^0(X, I) \le h^0(X, E) = 0$. Thus we have
\[
D_1 \ \le \ (\deg I + h)(n-r-1) + h + r(g-1).
\]
Secondly, $\dim \Gr(2, F/H) = 2(n-r-2) + 1$ and 
\[
D_2:= \dim \Sec^{\frac{1}{2}(e+\tilde{e}-2h)} \Gr(2, F/H)  \ \le \ (e+\tilde{e}-2h)(n-r-1) -1.
\]
Thirdly, since $\text{rk}(\wedge^2 (F/H)) = \frac{1}{2} (n-r)(n-r-1)$ and $\deg (\wedge^2 (F/H)) = -(n-r-1)(e-h)$, we have
\[
\begin{split}
D_3 &:= \dim \Pi_j - h^1(\wedge^2 (F/H)) \\ &\le \ \dim \Pi_j - (n-r-1)(e-h) - \frac{(n-r)(n-r-1)}{2}(g-1).
\end{split}
\]
By Lemma \ref{liftingcrit} (3), for fixed $j$, the classes $[V]$ vary in a locus of dimension $(D_2 + 1) + D_3$.

Finally, the subbundle $I$ of $E$ varies in a Quot scheme whose dimension is $h^0 (X, \Hom(I, E/I))$. By Laumon \cite[Proposition 3.5]{Lau}, since $E$ is general, it is very stable. Thus by Lange--Newstead \cite[Lemma 3.3]{LaNe} we have $h^1 (X, \Hom(I, E/I)) = 0$ for all subbundles $I \subset E$. Thus we compute
\[
D_4:= h^0 (X, \Hom(I, E/I)) = -n \cdot \deg I - re - r(n-r)(g-1).
\]

We now compare the sum $S_1 = D_1 + (D_2 + 1) + D_3 + D_4$ with the dimension $S_2$ of the fiber $\left( \tau_e \circ \pi_e \right)^{-1} (\bar{E})$. 
We have
\[
S_2 := \dim \Pi_j + \dim H^1(X, E) = \dim \Pi_j + e + n(g-1).
\]
Computing, we find that $S_1 < S_2$ if 
\begin{equation} \label{ineq}
h -(r+1)e + (n-r-1)\tilde{e} - (r+1) \deg I \ < \ \frac{(n-r)(n+r+1)}{2}(g-1).
\end{equation}
From Corollary \ref{capcor} it follows that $-\deg I \le h$. Furthermore, $\tilde{e} \leq e$ by hypothesis. Thus by Lemma \ref{liftingcrit} (2), we have $h \le \frac{1}{2}(e + \tilde{e}) \leq e$. Applying these inequalities, we see that the expression on the left is bounded above by
\[ e -(r+1)e + (n-r-1)e + (r+1)e 
\ = \ e(n - r). \]
Thus the required inequality (\ref{ineq}) would follow from
\begin{equation} e \ < \ \frac{(n+r+1)(g-1)}{2}. \label{mainineq} \end{equation}
This is satisfied for $r \ge 0$ by the hypothesis $e < \frac{1}{2}(n+1)(g-1)$.

Now we need only to deal with the case when $r = n-1 \ge 1$. Since $E$ is general,
\[
-(n-1)e - n \cdot \deg I \ \ge \ (n-1)(g-1).
\]
Combining with the inequality $-\deg I \le h \le e$, we get $(n-1)(g-1) \le e$. From the hypothesis $e < \frac{1}{2}(n+1)(g-1)$, we have $n < 3$. 
Thus it remains to consider the case when $n = 2$ and $r = \rank I = 1$.  In this case, we claim that given a general $E$, the extensions $j \in H^1 (X, E)$ admitting a diagram of the form (\ref{IH}) are special. Indeed, from the previous computations, the dimension of the locus of such extensions is bounded by $D_1 + D_4$, which is computed in this case as:
\[
D_1 + D_4 \ \le \ [ h + (g-1)] \ + \  [-2  \deg I - e - (g-1)] \ = \ h-e - 2 \deg I \ \le \ 2e.
\]
On the other hand, $h^1(X, E) = e + 2(g-1)$. From the assumption $e < \frac{3}{2}(g-1)$, we get
\[
D_1 + D_4 \ \le \ 2e \ < \ e + \frac{3}{2}(g-1) \ < \ h^1(X, E).
\]
This confirms the claim.
\end{proof}

Now we consider the consequences of Proposition \ref{finite}. By Proposition \ref{generalstable}, for $e < \frac{1}{2}(n+1)(g-1)$, a general point $V \in \A_e$ represents a stable orthogonal bundle. 
As discussed in \S \ref{stability}, there is a rational moduli map $\alpha_e \colon \A_e \dashrightarrow MO_X(2n+1)$.

\begin{theorem}  \label{thm3}
\begin{enumerate}
\renewcommand{\labelenumi}{(\arabic{enumi})}
\item For each even number $t$ with $0 < t < (n+1)(g-1)$, the stratum $MO_X(2n+1;t)$ is nonempty and irreducible of dimension equal to $\frac{1}{2} n(3n+1)(g-1) + \frac{1}{2}nt$.
\item For $(n+1)(g-1) \le t \le (n+1)(g-1) + 3$, the stratum $MO_X(2n+1;t)$ is dense in a component $MO_X(2n+1)^\pm$.
\end{enumerate}
\end{theorem}
\begin{proof}
(1) For $t = 2e$ in the range $0 < t < (n+1)(g-1)$, the stratum $MO_X(2n+1;t)$ is nonempty, by Proposition \ref{finite}. It contains the image $\alpha_e(\A_e)$, which is irreducible. To show the irreducibility of $MO_X(2n+1;t)$, we need to show that $\alpha_e(\A_e)$ is dense in $MO_X(2n+1;t)$. 
Any bundle $V_0 \in MO_X(2n+1;t)$ contains a maximal Lagrangian subbundle $E_0$ of degree $-e$, which might be unstable. Considering a versal deformation of $E$, we can find a one-parameter family $\{ {E}_\lambda \}$ containing $E_0$ with general $E_\lambda$ being stable. Along this family, we can build a family of orthogonal bundles admitting $E_\lambda$ as maximal Lagrangian subbundles, using the parameter space in Proposition \ref{taut}. Since a general orthogonal bundle in this family is 
represented in $\A_e$, the bundle $V_0 \in MO_X(2n+1)$ lies on the closure of $\alpha_e (\A_e)$, as required.

For $t = 2e < ({n+1}) (g-1)$, the map $\alpha_e$ is generically finite, by Proposition \ref{finite}. Therefore, $MO_X(2n+1;t)$ has the same dimension as $\A_e$. By the computation in Proposition \ref{taut} (3), this is $\frac{1}{2} n(3n+1)(g-1) + \frac{1}{2}nt$. 

(2) For $t < ({n+1})(g-1)$ we have $\dim MO_X(2n+1;t) < \dim MO_X(2n+1)$, while we know that $t(V) \le (n+1)(g-1) + 3$ by Proposition \ref{oddbound}. Therefore, the two strata corresponding to even numbers in the range $(n+1)(g-1) \le t \le (n+1)(g-1) + 3$ must be nonempty and dense in the components $MO_X(2n+1)^\pm$.
\end{proof}

\begin{remark} In particular, the last statement shows a general bundle represented in either of the corresponding parameter spaces $\A_e$ must be stable. \qed
\end{remark}

\begin{cor} \label{thm2}
For any orthogonal bundle $V$ of rank $2n+1$, we have $t(V) \le (n+1)(g-1)+3$. This bound is sharp in the sense that the two even numbers $t$ with $(n+1)(g-1) \le t \le (n+1)(g-1) + 3$ correspond to the values of $t(V)$ for a general $V \in MO_X(2n+1)^\pm$. 
\end{cor}

From the computation in Proposition \ref{finite} we now deduce a statement for generic orthogonal bundles of odd rank, analogous to Lange--Newstead \cite[Theorem 2.3]{LaNe} for vector bundles and \cite[Theorem 4.1 (3)]{CH3} for symplectic bundles:

\begin{cor} As before, write $\deg E = -e$. Suppose that $g \geq 5$ and
\[ \frac{(n+1)(g-1)}{2} \ \le \ e \ \leq \ \frac{(n+1)(g-1) + 3}{2}. \]
Then a general orthogonal extension $0 \to E \to V \to F^* \to 0$ admits no Lagrangian subbundle of degree $\ge -e$ intersecting $E$ in generically positive rank. \end{cor}
\begin{proof}
We must exclude the lifting of an $\tE$ of degree $\ge -e$ for which $I = E \cap \tE$ has rank $r \geq 1$. By (\ref{mainineq}), this would follow from
\[ \frac{(n+1)(g-1) + 3}{2} \ < \ \frac{(n+2)(g-1)}{2}, \]
which is satisfied for all $g \ge 5$. \end{proof}

\begin{remark} \label{Hirsch}
According to Hirschowitz \cite{Hir} (see also \cite{CH1}), every vector bundle $V$ of rank $2n+1$ and degree $0$ has a subbundle of rank $n$ and degree $-f$, where
\begin{equation} \label{hir}
f \ \le \ \left\lceil \frac{n(n+1)(g-1)}{2n+1} \right\rceil. 
\end{equation}
By direct computation, one can check that for $n \ge 1$ and $g \ge 2$, the bound in (\ref{hir}) is strictly smaller than $\frac{1}{2} t(V)$ for a general $V \in MO_X(2n+1)$, except in the following cases:
\begin{enumerate}
\renewcommand{\labelenumi}{(\roman{enumi})}
\item $g=2$, $n$ is odd and $t(V) = n+1$
\item $g=2$, $n$ is even and $t(V) = n+2$
\item $g=3$, $t(V) = 2(n+1)$
\item $g=4$, $n$ is odd and $t(V) = 3(n+1)$
\end{enumerate}
As mentioned in the introduction, this implies apart from in these cases, a general orthogonal bundle of rank $2n+1$ has the property that no maximal subbundle of rank $n$ is Lagrangian. \qed
\end{remark}

\noindent We conclude this section with a result on the topology of the strata. We recall the following statement \cite[Theorem 1.3 (2)]{CH3} for orthogonal bundles of rank $2n+2$:
\begin{prop} For each $t < (g-1)(n+1)$, the stratum $MO_X (2n+2; t)$ is contained in the closure of $MO_X (2n+2;t+4)$. \qed \label{closureevenrank} \end{prop}

Using Proposition \ref{Lagsubbs}, we will now deduce an analogous statement for orthogonal bundles of odd rank. We define a map $\Psi \colon MO_X (2n+1) \to MO_X (2n+2)$ by $\Psi (V) = V \perp \ox$. This is clearly an injective morphism, and $\Psi(V)$ is a stable orthogonal bundle if $V$ is stable.

\begin{theorem} For each $t < (g-1)(n+1)$, the stratum $MO_X (2n+1;t)$ is contained in the closure of $MO_X (2n+1; t+4)$. \end{theorem}
\begin{proof} Let $V \in MO_X (2n+1)$. By Proposition \ref{Lagsubbs} (3) we see that $V$ admits a maximal Lagrangian subbundle of degree $-e$ if and only if $V \perp \ox$ admits a maximal Lagrangian subbundle of degree $-e$. Therefore,
\[ MO_X (2n + 1;t) \ \cong \ \Psi \left( MO_X (2n + 1) \right) \ \cap \ MO(2n + 2; t) \]
for each $t$. Since all the spaces under consideration are constructible sets, we deduce from Proposition \ref{closureevenrank} that the stratum $MO_X (2n+1;t)$ is contained in the closure of $MO_X (2n+1;t+4)$ for each $t < (n+1)(g-1)$, as required. \end{proof}
%

\section{Maximal Lagrangian subbundles}

By Proposition \ref{finite}, the points of a general fiber $\alpha_e^{-1} (V)$ in $\A_e$ correspond to the maximal Lagrangian subbundles $ E \subset V$ which are stable as vector bundles. For $t = 2e < (n+1)(g-1)$, we already know by Proposition \ref{finite} (1) that a general $V \in MO_X(2n+1;t)$ has a unique maximal Lagrangian subbundle of degree $-e$. In this section we will compute the dimension of the space $M(V)$ of maximal Lagrangian subbundles for a general $V \in MO_X(2n+1)$ with $t(V) \ge (n+1)(g-1)$. We first observe:
\begin{prop} \label{max}
Suppose $E \subset V$ is a Lagrangian subbundle. 
Then the tangent space of $M(V)$ at $E$ is isomorphic to $H^0(X, \wedge^2 (E^\perp)^*)$.
\end{prop}
\begin{proof}
To give a Lagrangian subbundle $E \subset V$ is equivalent to giving a global section $\sigma \colon X \to \OG(n, V)$ of the orthogonal Grassmannian bundle $\OG(n, V)$ over $X$. A tangent vector to $M(V)$ at $E$, corresponding to an infinitesimal deformation of $E$ in $M(V)$, is equivalent to a global section of the pullback by $\sigma$ of the tangent bundle $T_{\OG(n, V)}$. This is the bundle $T \to X$ with $T_x \cong T_{E_x}\OG(n, V_x)$. By Lemma \ref{OG} (4), we have
\[
T_{E_x}\OG(n, V_x) \ \cong \ T_{F'_x} \OG(n+1, W_x) \ \cong \ \wedge^2 \left( F'_x \right)^* ,
\]
where $W \cong V \perp \ox$ and $F' \subset W$ is a Lagrangian subbundle isomorphic as a vector bundle to $E^\perp \subset V$. Therefore, the bundle $T \to X$ can be identified with $\wedge^2 (F')^*$. \end{proof}

By the proposition, the dimension of a general fiber $\alpha_e^{-1}(V)$ is given by $h^0(X, \wedge^2 F^*)$, where $F$ is general in $S_e \subset U_X(n+1, -e)$.  

\begin{lem} \label{h0F*}
Let $F$ be a general bundle in $S_e \subset U_X(n+1, -e)$. Then
\[
h^0(\wedge^2 F^*) = 
\begin{cases}
0 & \text{ if } e \le \frac{1}{2}(n+1)(g-1), \\
ne - \frac{1}{2} n(n+1)(g-1) & \text{ if } e > \frac{1}{2}(n+1)(g-1).
\end{cases}
\]
\end{lem}
\begin{proof}
Firstly, suppose that $e \le \frac{1}{2}(n+1)(g-1)$, so that $\mu(\wedge^2 F^*) \le g-1$. If $F$ were general in $U_X(n+1, -e)$, then we would have $h^0(X, \wedge^2 F^*) = 0$ by the variant \cite[Lemma A.1]{CH2} of Hirschowitz' lemma. However, in this case $F$ belongs to the locus $S_e \subset U_X(n+1, -e)$. Since $S_e$ is irreducible, it suffices to find one $F \in S_e$ satisfying $h^0(X, \wedge^2 F^*) = 0$. We will do this for $e = \frac{1}{2}(n+1)(g-1)$, since the lower degree cases are easier. 

For $n=1$, we need to find a rank 2 bundle $F^*$ of degree $g-1$, so that $h^0(X, \det F^*) = 0$ and $h^0(X, F^*) >0$. Let $L$ be a general line bundle of degree $g-1$ with $h^0(X, L) = 0$. Then a general extension $0 \to \ox \to F^* \to L \to 0$ satisfies the requirements.

For $n \ge 2$, consider a polystable bundle $F^*$ of the form $G \oplus H$, where $G \in U_X(2, g-1)$ and $H \in U_X \left( n-1, \frac{1}{2}(n-1)(g-1) \right)$. Note that 
\[
\wedge^2 F^* \ \cong \ (\det G) \oplus (\wedge^2 H) \oplus (G \otimes H).
\]
By Lange--Narasimhan \cite{LN}, the bundle $G$ has finitely many maximal line subbundles of degree zero. Choose a general such $G$ with a maximal line subbundle $M$. Then by the Hirschowitz lemma \cite[\S 4.6]{Hir} we have $h^0(X, G \otimes H) = 0$ for general $H$. Now put $\tilde{G} = G \otimes M^{-1}$ and $\tilde{H} = H \otimes M$. Then $h^0(X, \tilde{G}) >0$ and $h^0(X, \tilde{G} \otimes \tilde{H}) = 0$. The vanishing of $H^0(X, \wedge^2 \tilde{H})$ follows again from \cite[Lemma A.1]{CH2}. 

Now let $\tilde{F}^* = \tilde{G} \oplus \tilde{H}$. Since $h^0(X, \tilde{G}) > 0$, we have $h^0(X, \tilde{F}^*) > 0$.
To obtain the vanishing of $h^0(X, \wedge^2 \tilde{F}^*)$, we must show that $h^0(X, \det \tilde{G}) = 0$. Since $G$ has rank two, $\det G \otimes M^{-2} \cong \Hom(M, G/M)$. By generality of $M$ and $G$, there are no deformations of $M$ in $G$, so $h^0 (X, \Hom(M, G/M)) = 0$. 

The statement for $e > \frac{1}{2}(n+1)(g-1)$ is equivalent to the vanishing of $h^1(X, \wedge^2 F^*)$. By Serre duality, this is in turn equivalent to the vanishing of $h^0(X, \wedge^2 (F \otimes \kappa))$ for a theta characteristic $\kappa$. But if $e > \frac{1}{2}(n+1)(g-1)$, then $\deg (F \otimes \kappa) < \frac{1}{2}(n+1)(g-1)$ and we reduce to the case proven above.
\end{proof}

\begin{theorem} \label{thm4}
Let $V$ be a general orthogonal bundle in $MO_X(2n+1;t)$. If $t <(n+1)(g-1)$ then $V$ has a unique maximal Lagrangian subbundle. If $t = (n+1)(g-1)$ (resp. $t > (n+1)(g-1)$) is even, $V $ has a finite (resp. infinite) number of maximal Lagrangian subbundles.
\end{theorem}
\begin{proof}
The uniqueness for $t <(n+1)(g-1)$ has already been proven in Proposition \ref{finite}. The finiteness for $t = (n+1)(g-1)$ comes from Proposition \ref{max} and Lemma \ref{h0F*}. The dimension of the space $M(V)$ of maximal Lagrangian subbundles for a general $V \in MO_X(2n+1)$ is also given by Lemma \ref{h0F*}.
\end{proof}

As in the even rank case \cite[\S 5.4]{CH3}, for a general $V \in MO_X(2n+1)^\pm$ the invariants $t(V)$ and $\dim M(V)$ depend on the congruence class of $N := (n+1)(g-1)$ modulo 4. By Lemma \ref{h0F*}, we have
\[
\dim M(V) \ = \ ne-\frac{1}{2} n(n+1)(g-1) \ = \ \frac{1}{2} nt - \frac{1}{2}n(n+1)(g-1) \ = \ \frac{n}{2}(t(V) - N).
\]
We deduce the following table, by analogy with that in \cite[\S 5.4]{CH3}:
\[ N \equiv 0 \mod 4: \quad \quad \begin{array}{c|c|c} t(V) & \hbox{Component} & \dim M(V) \\ \hline
N & MO_{X}(2n+1)^{+} & 0 \\
N + 2 & MO_{X}(2n+1)^{-} & \hbox{$n$} \end{array} \]

\[ N \equiv 1 \mod 4: \quad \quad \begin{array}{c|c|c} t(V) & \hbox{Component} & \dim M(V) \\ \hline
N + 1 & MO_{X}(2n+1)^{-} & \hbox{$n/2$} \\
N + 3 & MO_{X}(2n+1)^{+} & \hbox{$ {3n}/{2}$} \end{array} \]

\[ N \equiv 2 \mod 4: \quad \quad \begin{array}{c|c|c} t(V) & \hbox{Component} & \dim M(V) \\ \hline
N & MO_{X}(2n+1)^{-} & \hbox{0} \\
N + 2 & MO_{X}(2n+1)^{+} & \hbox{$n$} \end{array} \]

\[ N \equiv 3 \mod 4: \quad \quad \begin{array}{c|c|c} t(V) & \hbox{Component} & \dim M(V) \\ \hline
N + 1 & MO_{X}(2n+1)^{+} & \hbox{$n/2$} \\
N + 3 & MO_{X}(2n+1)^{-} & \hbox{$3n/2$} \end{array} \]
\vspace{0.5cm}

\section*{Acknowledgement}
 Insong Choe acknowledges that this paper was written as part of Konkuk University's research support program for its faculty on sabbatical leave in 2013. He would like to thank the Department of Mathematical Sciences in SNU for the support on his visiting during 2013. He also acknowledges the support of KIAS as an associate member in 2013.

\noindent \footnotesize{Department of Mathematics, Konkuk University, 1 Hwayang-dong, Gwangjin-Gu, Seoul 143-701, Korea.\\
E-mail: \texttt{ischoe@konkuk.ac.kr}\\
\\
H\o gskolen i Oslo og Akershus, Postboks 4, St. Olavs plass, 0130 Oslo, Norway.\\
E-mail: \texttt{george.hitching@hioa.no}}

\end{document}